\documentclass[12pt]{amsart}
\usepackage{amsmath, amssymb, amsfonts}
\usepackage[all]{xypic}
\usepackage[pdftex]{graphicx}
\usepackage[usenames]{color}
\usepackage{faktor}
\usepackage{tikz-cd}
\usepackage{stmaryrd}
\usepackage[utf8]{inputenc}
\usepackage{bbold}
\usepackage{hyperref}

\theoremstyle{plain}
\newtheorem{theorem}{Theorem}[section]
\newtheorem{lemma}[theorem]{Lemma}

\newtheorem{prop}[theorem]{Proposition}

\theoremstyle{remark}

\newtheorem{remark}[theorem]{Remark}

\newtheorem*{note*}{Note}
\newtheorem*{remark*}{Remark}
\newtheorem*{example*}{Example}

\theoremstyle{definition}
\newtheorem*{definition*}{Definition}
\newtheorem*{hypothesis*}{Hypothesis}
\newtheorem*{assumptions*}{Assumptions}
\newtheorem{definition}[theorem]{Definition}

\DeclareMathOperator{\Gal}{Gal}
\DeclareMathOperator{\aug}{aug}
\DeclareMathOperator{\Hom}{Hom}
\DeclareMathOperator{\Map}{Map}
\DeclareMathOperator{\Cone}{Cone}
\DeclareMathOperator{\ctp}{\widehat{\otimes}}
\DeclareMathOperator{\cd}{cd}
\def\ram{{\rm\mathop{ram}}}
\def\Ind{{\rm\mathop{Ind}}}

\newcommand*{\ZZ}{\mathbb{Z}}
\newcommand*{\QQ}{\mathbb{Q}}
\newcommand*{\LL}{\mathbb{L}}

\numberwithin{equation}{section}

\title{On natural complexes in Iwasawa theory}

\author{Antonio Mejías Gil}
\address{
	Universit\"{a}t Duisburg-Essen\\
	Fakult\"{a}t f\"{u}r Mathematik\\
	Thea-Leymann-Str. 9\\
	45127 Essen\\
	Germany}
\email{anmegi.95@gmail.com}

\author{Andreas Nickel}
\address{
	Universit\"{a}t der Bundeswehr M\"unchen\\
	Institut f\"ur Theoretische Informatik, Mathematik und Operations Research\\
	Werner-Heisenberg-Weg 39\\
	85579 Neubiberg\\
	Germany}
\email{andreas.nickel@unibw.de}
\urladdr{https://www.unibw.de/timor/mitarbeiter/univ-prof-dr-andreas-nickel}

\date{Version of 1st February 2024}

\begin{document}

    \maketitle
    
    \begin{abstract}
    	We consider two natural complexes that appear in recent formulations of equivariant Iwasawa main conjectures
    	for extensions of not necessarily totally real fields. We show that both complexes are isomorphic in the
    	derived category of Iwasawa modules if the weak Leopoldt conjecture holds for the relevant
    	$\ZZ_p$-extension.
    \end{abstract}

    \section*{Introduction}
    
        The study of class numbers of number fields constitutes one of the central aims of modern number theory, motivated by results such as Kummer's proof of Fermat's last theorem in the case of regular primes. One of the most fruitful approaches to this pursuit has been the combination of arithmetic and analytic techniques, exemplified by the celebrated analytic class number formula
        \[
            \lim_{s \to 1} \ (s - 1) \zeta_K(s) = \frac{2^{r_1} (2\pi)^{r_2} h_K R_K}{w_K \sqrt{|d_K|}}.
        \]
        Here, the left-hand side belongs to the realm of analytic number theory, whereas the right-hand one consists of primarily algebraic invariants -- most notably, the elusive class number $h_K$ itself. The so-called equivariant Tamagawa Number Conjecture (eTNC) of Burns and Flach \cite{bf_etnc} vastly generalises this algebraic-analytic connection to the setting of \textit{motives}, and therefore also encompasses the deep Birch and Swinnerton-Dyer conjecture on ranks of elliptic curves.
        
        In the 1960s, Iwasawa pioneered an approach to the study of such relations in the case of class groups consisting of the consideration of the \textit{asymptotic} behaviour of class numbers along infinite towers, rather than at a fixed \textit{finite level} (i.e. a single number field) as before. His original \textit{Main Conjecture} concerned class-group behaviour along the  cyclotomic $\ZZ_p$-extension of abelian extensions of $\QQ$, a setting which was quickly extended to an arbitrary totally real base field $K$. The Kubota-Leopoldt $p$-adic $L$-function, and its generalisation independently constructed by Pierrette Cassou-Nogu\`es \cite{cassou-nogues}, Deligne and Ribet \cite{deligne_ribet} and Barsky \cite{barsky}, play the analytic role in the conjecture. The first major breakthrough in this space came in 1984, when Mazur and Wiles \cite{mazur_wiles} settled the (character-wise) abelian case with base field $K = \QQ$. Wiles \cite{wiles} then extended these techniques and gave a proof in the general totally real case six years later.
        
        In the last two decades, the focus has been put on \textit{equivariant} Main Conjectures, where one predicts an algebraic-analytic relation that simultaneously encompasses all characters of the Galois group under consideration, rather than postulating individual \textit{character-wise} relations as in the original work of Iwasawa. One of the first such examples is Burns and Greither's equivariant conjecture from \cite{burns_greither}, which they proved for abelian extensions of $\QQ$ and used to deduce corresponding cases of the eTNC. In 2004, Ritter and Weiss \cite{rwii} formulated an equivariant Main Conjecture for arbitrary totally real number fields with no abelianity requirement, which they proved in the so-called $\mu = 0$ case in a series of articles culminating in \cite{rw_on_the}. Independently, Kakde \cite{kakde_mc} formulated and proved a Main Conjecture in the same setting which additionally allowed for $\ZZ_p$-extensions of higher rank. Venjakob \cite{venjakob_on_rw} and the second named author \cite{nickel_plms} independently proved Kakde's conjecture to be equivalent to that of Ritter and Weiss
        in the rank $1$ case. Another important family of Main Conjectures concerns imaginary quadratic base fields, an area initiated by Rubin which has seen important cases being settled by Rubin himself \cite{MR0952288, MR1079839}, Bley \cite{bley} and, more recently, Bullach and Hofer \cite{bullach_hofer}.
        
        The Main Conjecture of Ritter and Weiss and of Kakde is formulated in the language of algebraic $K$-theory and contains as its central algebraic object a certain \textit{perfect complex} which defines a class in the relevant relative $K$-group. However, the non-vanishing of the rank of the classical $S$-ramified Iwasawa module $X_S$ when the base field $K$ is not totally real, renders this approach unsuitable for direct generalisation. A different machinery was employed by Burns, Kurihara and Sano \cite{bks} in 2017, when they postulated a Main Conjecture for \textit{abelian} extensions of arbitrary number fields using the language of determinant functors and exterior powers. The algebraic complex of Ritter and Weiss is replaced by a different one such that the positive rank of $X_S$ is no longer an obstacle. Until recently, exterior powers were only defined for commutative rings, which complicated generalisation of the latter conjecture. The first named author's PhD thesis \cite{thesis}, under the supervision of the second named author, formulated a new Main Conjecture which, in a suitable sense, generalises these two conjectures to remove both the total-realness and abelianity assumptions. It is important to note that Burns and Sano \cite{burns_sano} independently extended their and Kurihara's conjecture to the non-abelian case by introducing so-called \textit{reduced exterior powers} -- which provide a non-commutative analogue of classical exterior powers. The exact connection between this conjecture and that of \cite{thesis} is unclear as of yet, but it stands to reason that both are largely equivalent, inasmuch as they are both natural  generalisations of \cite{bks} to the non-abelian setting.
        
        In this article, we present and generalise the construction of the complex featured in the conjecture formulated in \cite{thesis}, which largely relies on explicit maps between classical Iwasawa modules; and then prove that the resulting complex coincides with that used in the conjecture of Burns, Kurihara and Sano whenever the weak Leopoldt conjecture holds for the relevant $\ZZ_p$-extension. Most notably, this is the case whenever the latter is the cyclotomic one. Coincidence here is to be understood as isomorphism in the natural derived category. A full comparison of the two conjectures can be found in \cite[\S 4.1]{thesis}. 
        Another thing we won't cover here is the relation to the complex of Ritter and Weiss. This, as well as the relation between the corresponding conjectures, is explained in detail in \cite[\S 4.2]{thesis}.
        
        \subsection*{Acknowledgements}
        The authors wish to express their gratitude towards Daniel Macías Castillo and Dominik Bullach for the organisation of the excellent conference in Madrid. The first named author would also like to thank the University of Duisburg-Essen, and in particular the Essen Seminar for Algebraic Geometry and Arithmetic, for the stipend that enabled the doctoral project from which this article stems. Finally, the authors are indebted to a number of people for many fruitful mathematical conversations, among which we would like to single out Werner Bley, Dominik Bullach, Bence Forrás, Sören Kleine, Alexandre Maksoud, Takamichi Sano and Otmar Venjakob.

    \section{Construction of the main complex}
           
%
        \subsection{Setup and notation}
        Let $p$ be a prime and let
        $G$ be a profinite group.
        The complete group algebra of $G$ over $\ZZ_p$ is 
        \[
        \Lambda(G) := \ZZ_{p}\llbracket G \rrbracket = \varprojlim \ZZ_{p}[G/N],
        \]
        where the inverse limit is taken over all open normal subgroups 
        $N$ of $G$. Then $\Lambda(G)$ 
        is a compact $\ZZ_p$-algebra and we denote
        the kernel of the natural augmentation map
        $\aug \colon \Lambda(G) \twoheadrightarrow \ZZ_p$ by $\Delta(G)$.
        If $M$ is a (left) $\Lambda(G)$-module we let
        $M_{G} := M / \Delta(G) M$ be the module of coinvariants
        of $M$. This is the maximal quotient module of $M$ with trivial
        $G$-action. 
        Similarly, we denote the maximal submodule of $M$ upon which
        $G$ acts trivially by $M^{G}$.
        For any topological $G$-module $M$
        we write $R\Gamma(G, M)$ for the complex
        of continuous cochains of $G$ with coefficients in $M$.
        We denote the absolute Galois group of a field $F$ by $G_F$.
        For any noetherian ring $R$ we write $\mathcal{D}(R)$ for the
        derived category of (left) $R$-modules. We recall that a complex
        of $R$-modules is called \textit{perfect} if it is isomorphic
        in $\mathcal{D}(R)$ to a bounded complex of finitely generated
        projective $R$-modules.
        
        Let $K$ be a number field and let $\Sigma$ be a finite set of places
        of $K$ containing the set $S_{\infty}$ of archimedean places
        of $K$ and the set $S_p$ of all $p$-adic places of $K$.
        Let $L_\infty$ be a one-dimensional $p$-adic Lie extension of $K$ containing a $\ZZ_p$-extension $K_\infty$ of $K$. 
        We assume that the extension $L_{\infty}/K$ is unramified
        outside $\Sigma$, which can always be achieved by enlarging $\Sigma$
        if necessary. If $p = 2$, we in addition
        require $K$ to be totally imaginary.
        
        Let $M_{\Sigma}$ denote the maximal pro-$p$-extension
        of $L_{\infty}$ which is unramified outside $\Sigma$. We put 
        $G_{\Sigma} := \Gal(M_{\Sigma}/K)$ and $H_{\Sigma} := \Gal(M_{\Sigma}/L_{\infty})$.
        Since $K$ is totally imaginary
        if $p=2$, the cohomological $p$-dimension of $G_{\Sigma}$ equals
        $2$ in all cases by \cite[Proposition 10.11.3]{nsw}
        (note that our definition of $G_{\Sigma}$ follows
        \cite[Chapter XI, \S 3, p.739]{nsw}, but slightly
        differs from the profinite group $G_{\Sigma}$ considered
        in \cite[Chapter X, \S 11]{nsw}; however,
        the proof of \cite[Lemma 5.3]{jann} shows that both
        groups have the same cohomological $p$-dimension).

        Letting $H = \Gal(L_\infty/K_\infty)$ (which is finite), $\mathcal{G} = \Gal(L_\infty/K)$ and $\Gamma_K = \Gal(K_\infty/K) \simeq \ZZ_p$, it is well known that the short exact sequence $H \hookrightarrow \mathcal{G} \twoheadrightarrow \Gamma_K$ splits and one may choose an open central subgroup $\Gamma \simeq \ZZ_p$ of $\mathcal{G}$ (see for instance \cite[\S 1]{rwii}). We fix such a choice. 
        Then $L_\infty$ is a $\ZZ_p$-extension of 
        $L := L_\infty^\Gamma$ and every intermediate layer $L_n$ is a finite Galois extension of $K$. We denote the corresponding group by $\mathcal{G}_n = \Gal(L_n/K)$.

        
        For every place $v$ of $K$, fix a distinguished prolongation $v^c$ to the algebraic closure $K^c = \QQ^c$ of $K$.
        Given $F \subseteq K^c$, we denote by $v(F)$
        the place of $F$ below $v^c$ and, if $F$ is Galois over $K$,
        we denote the decomposition group at $v(F)$ simply by
        $\Gal(F/K)_v$. Different choices of prolongations lead to conjugate groups, which has no influence on any of the results presented.
        If $S$ is a finite set of places of $K$, we write $S(F)$ for the
        set comprising all places of $F$ above places in $S$.
        
        We will henceforth assume that every non-archimedean place of $K$ has open decomposition group in $\mathcal{G}$ and therefore splits into finitely many places in $L_\infty$. In addition, we will often,
        but not always, assume that the \textit{weak Leopoldt conjecture} holds for $L_\infty/L$, i.e.\ we assume that $H_2(H_{\Sigma}, \ZZ_p)$ vanishes.
        Note that both conditions are indeed satisfied if $K_{\infty}$
        is the cyclotomic $\ZZ_p$-extension of $K$ by
        \cite[Exercise 13.2]{wash} and \cite[Theorem 10.3.25]{nsw}, respectively.
       
    \subsection{Local and global complexes}
    \label{subsec:local_and_global_complexes}
        
        The main complex arises as the shifted cone of
        a morphism from local to global complexes of $\Lambda(\mathcal{G})$-modules, all of which result from an application of the so-called \textit{translation functor}  introduced by Gruenberg \cite{gruenberg} and expanded upon by Ritter and Weiss \cite{rw_tate_sequence}. We now review some of the details.
        
        Let $X_{\Sigma} = H_{\Sigma}^{ab}$ be maximal abelian quotient of $H_{\Sigma}$, i.e.\
        $X_{\Sigma}$ is the Galois group over $L_{\infty}$ of the 
        maximal abelian pro-$p$-extension of $L_{\infty}$ that is unramified
        outside $\Sigma$. This is a natural finitely generated
        $\Lambda(\mathcal{G})$-module.
        If we assume weak Leopoldt, it has no nontrivial finite
        $\Lambda(\Gamma)$-submodules and $\Lambda(\Gamma)$-rank $r_2(L)$,
        the number of complex places of $L$, by \cite[Theorem 10.3.22]{nsw}.
        Consider the short exact sequence of $\Lambda(G_{\Sigma})$-modules
        \begin{equation}
        \label{eq:augmentation_sequence_global}
            0 \to \Delta(G_{\Sigma}) \to \Lambda(G_{\Sigma}) \xrightarrow{\aug} \ZZ_p \to 0.
        \end{equation}
        Since $H_1(H_{\Sigma}, \Lambda(G_{\Sigma}))$ vanishes and
        $H_1(H_{\Sigma}, \ZZ_p)$ naturally identifies with
        $X_{\Sigma}$, taking $H_{\Sigma}$-coinvariants of sequence
        \eqref{eq:augmentation_sequence_global}
        yields a four-term exact sequence of $\Lambda(\mathcal{G})$-modules
        \begin{equation}
        \label{eq:complex_sequence_global}
            0 \to X_{\Sigma} \to \Delta(G_{\Sigma})_{H_{\Sigma}} \to \Lambda(\mathcal{G}) \to \ZZ_p \to 0.
        \end{equation}
         We set $Y_{\Sigma} := \Delta(G_{\Sigma})_{H_{\Sigma} }$. With this notation, we regard the middle arrow in sequence 
         \eqref{eq:complex_sequence_global} as a complex $\mathcal{T}_{\Sigma}^{\bullet} = [Y_{\Sigma} \to \Lambda(\mathcal{G})]$ of $\Lambda(\mathcal{G})$-modules concentrated in degrees $0$ and $1$. Then we have that $H^0(\mathcal{T}_{\Sigma}^{\bullet}) = X_{\Sigma}$ and $H^1(\mathcal{T}_{\Sigma}^{\bullet}) = \ZZ_p$.
         We will refer to $\mathcal{T}_{\Sigma}^{\bullet}$ as the 
         \textbf{global complex}.
        
        Turning now our attention to the local side, let $v$ first be a \textit{non-archimedean} place of $K$. We denote the completion
        of $K$ at $v$ by $K_v$ and regard the decomposition group $G_{L_{\infty},v}$ of $G_{L_{\infty}}$ at our distinguished prolongation $v^c$ as a subgroup of $G_{K_v}$.
        It is normal and the quotient identifies with
        the decomposition group $\mathcal{G}_v$.
        
        A construction analogous to the global case can now be carried out by taking {$G_{L_\infty, v}$-coinvariants} on the exact sequence
        \begin{equation}
        \label{eq:augmentation_sequence_local}
            0 \to \Delta(G_{K_v}) \to \Lambda(G_{K_v}) \to \ZZ_p \to 0,
        \end{equation}
        resulting in a four-term exact sequence of $\Lambda(\mathcal{G}_v)$-modules
        \begin{equation}
        \label{eq:complex_sequence_local}
            0 \to G_{L_{\infty, v}}^{ab}(p) \to \Delta(G_{K_v})_{G_{L_\infty, v}} \to \Lambda(\mathcal{G}_v) \to \ZZ_p \to 0.
        \end{equation}
        We denote the first term (the maximal abelian 
        pro-$p$-quotient of $G_{L_{\infty, v}}$) by $X_v$, 
        and the second one by $Y_v$.
        
        Finally, if $v$ is an \textit{archimedean} place of $K$, the same steps as above can be taken. However, due to its simplicity, we opt instead to give a more explicit description of the archimedean situation, where two cases can be distinguished:
        \begin{itemize}
            \item{
                If $v$ is \textit{unramified} (i.e. completely split) in $L_\infty/K$, then $\Delta(\mathcal{G}_v)$ is trivial and we regard the augmentation sequence associated to $\mathcal{G}_v$ as a four-term exact sequence
                \begin{equation} \label{eq:archimedean-unramified}
                    0 \to X_v \to Y_v \to \Lambda(\mathcal{G}_v) \to \ZZ_p \to 0
                \end{equation}
                with $X_v = Y_v = 0$. This is always the case if $p = 2$ by assumption.
            }
            \item{
                If $v$ is \textit{ramified} in $L_\infty/K$, the augmentation sequence associated to $\mathcal{G}_v$ takes the form
                \begin{equation} \label{eq:archimedean-ramified}
                     0 \to \Delta(\mathcal{G}_v) \to \Lambda(\mathcal{G}_v) \to \ZZ_p \to 0
                \end{equation}
                and splits as $e^- \ZZ_p \hookrightarrow e^- \ZZ_p \oplus e^+ \ZZ_p \twoheadrightarrow e^+ \ZZ_p$, where $e^+ = (1 + \tau_{w_\infty})/2$ and $e^- = 1 - e^+$ are the central idempotents constructed from the complex conjugation $\tau_{w_\infty}$ generating $\mathcal{G}_v$. Note that 2 is invertible in $\ZZ_p$ in this case as $p \neq 2$ by assumption. Sequence \eqref{eq:archimedean-ramified} can again be regarded as a sequence of the form \eqref{eq:archimedean-unramified} with $X_v = 0$ and $Y_v = \Delta(\mathcal{G}_v) = e^- \ZZ_p$.
            }
        \end{itemize}
        
        
        The upshot is that every place $v$ of $K$ has an associated four-term exact sequence with some possibly trivial terms. In all cases we regard the middle arrow as a complex $\mathcal{L}_v^{\bullet} = [Y_v \to \Lambda(\mathcal{G}_v)]$ of $\Lambda(\mathcal{G}_v)$-modules concentrated in degrees 0 and 1, with $H^0(\mathcal{L}_v^{\bullet}) = X_v$ and $H^1(\mathcal{L}_v^{\bullet}) = \ZZ_p$ as defined above. We refer to it as the \textbf{local complex} at $v$.

    \subsection{The main complex}
    \label{sec:the_main_complex}
        Let us now define a morphism from the local complexes to the global one. For a fixed $v$, consider the continuous group homomorphism
        \begin{equation}
        \label{eq:galois_map_local_to_global}
            G_{K_v} \rightarrow G_K \rightarrow G_{\Sigma}
        \end{equation}
        (where the first arrow is injective and the second is surjective,
        but the composite, in general, is neither -- cf.\ \cite[Chapters IX and X]{nsw}). It induces a homomorphism of $\Lambda(\mathcal{G}_v)$-modules
        \[
            \alpha_v^0 \colon \faktor{\Delta(G_{K_v})}{\Delta(G_{L_{\infty, v}}) \Delta(G_{K_v})} \to \faktor{\Delta(G_{\Sigma})}{\Delta(H_{\Sigma})\Delta(G_{\Sigma})},
        \]
        where the codomain is $Y_{\Sigma}$ by definition and the domain coincides with $Y_v$ by an immediate computation (whether $v$ is archimedean or not). We also define $\alpha_v^1 \colon \Lambda(\mathcal{G}_v) \hookrightarrow \Lambda(\mathcal{G})$ as the canonical embedding.
        
        The commutativity of
        \begin{center}
            \begin{tikzcd}
                Y_v \arrow[d, "\alpha_v^0"] \arrow[r] & \Lambda(\mathcal{G}_v) \arrow[d, "\alpha_v^1"] \\
                Y_{\Sigma} \arrow[r]                & \Lambda(\mathcal{G})
            \end{tikzcd}
        \end{center}
        implies that the map $\alpha_v \colon \mathcal{L}_v^{\bullet} \to \mathcal{T}_{\Sigma}^{\bullet}$ given by $\alpha_v^0$ and $\alpha_v^1$ in degrees 0 and 1, respectively, defines a morphism of $\Lambda(\mathcal{G}_v)$-complexes.
        
        Let $\Ind_{\mathcal{G}_v}^\mathcal{G}$ be the compact induction functor $\Lambda(\mathcal{G}) \ctp_{\Lambda(\mathcal{G}_v)} -$, and denote by $\Ind_{\mathcal{G}_v}^\mathcal{G} \mathcal{L}_v^{\bullet}$ the result of applying said functor to $\mathcal{L}_v^{\bullet}$ degree-wise. In fact, for the case at hand $\Ind_{\mathcal{G}_v}^\mathcal{G} \mathcal{L}_v^{\bullet}$ coincides with $\Lambda(\mathcal{G}) \ctp_{\Lambda(\mathcal{G}_v)}^\LL \mathcal{L}_v^{\bullet}$.
        If $v$ is archimedean, this follows from the particular shape
        of the sequences \eqref{eq:archimedean-unramified} and
        \eqref{eq:archimedean-ramified}. If $v$ is a finite place,
        the Iwasawa algebra $\Lambda(\mathcal{G})$ is free of finite rank
        over $\Lambda(\mathcal{G}_v)$ by our running assumptions so that
        the (compact) induction functor is exact.

        Since $\mathcal{T}_{\Sigma}^{\bullet}$ is a complex of $\Lambda(\mathcal{G})$-(and not only $\Lambda(\mathcal{G}_v)$-)modules, induction yields a morphism of complexes $\Ind_{\mathcal{G}_v}^\mathcal{G} \alpha_v \colon \Ind_{\mathcal{G}_v}^\mathcal{G} \mathcal{L}_v^{\bullet} \to \mathcal{T}_{\Sigma}^{\bullet}$ of $\Lambda(\mathcal{G})$-modules.
        Let $S \subseteq \Sigma$ be a subset of $\Sigma$ containing
        $S_{\infty}$ and all places that ramify in $L_{\infty}$.
        We put $T := \Sigma \setminus S$ so that $\Sigma = S \cup T$
        is the disjoint union of $S$ and $T$.
         Letting $\mathcal{L}_S^{\bullet} = \bigoplus_{v \in S} \Ind_{\mathcal{G}_v}^\mathcal{G} \mathcal{L}_v^{\bullet}$ and
        \[
            \alpha_{S, T} = \sum_{v \in S} \Ind_{\mathcal{G}_v}^\mathcal{G} \alpha_v \colon \mathcal{L}_S^{\bullet} \to \mathcal{T}_{S \cup T}^{\bullet},            
        \]
        we are now in a position to define the main complex.
                
        \begin{definition}
            The \textbf{main complex} is the complex of $\Lambda(\mathcal{G})$-modules
            \[
                \mathcal{C}_{S, T}^{\bullet} = \Cone(\mathcal{L}_S^{\bullet} \xrightarrow{\alpha_{S, T}} \mathcal{T}_{S \cup T}^{\bullet})[-1].
            \]
        \end{definition}
        
        \begin{lemma}
        \label{lem:perfect}
        	Using the above notation the following hold:
        	\begin{enumerate}
        		\item 
        		For any place $v$ of $K$, the local complex $\mathcal{L}_v^{\bullet}$ is a perfect complex of $\Lambda(\mathcal{G}_v)$-modules. 
        		\item 
        		If the weak Leopoldt
        		conjecture holds, both
        		the global complex $\mathcal{T}_{S \cup T}^{\bullet}$ and the main complex $\mathcal{C}_{S, T}^{\bullet}$ are perfect complexes of $\Lambda(\mathcal{G})$-modules. 
        	\end{enumerate}
        \end{lemma}
        
        \begin{proof}
            For any short exact sequence of profinite groups $N \hookrightarrow G \twoheadrightarrow Q$ such that $H_2(N, \ZZ_p)$ vanishes, 
            and $G$ is topologically finitely generated and has cohomological $p$-dimension $\cd_p G \leq 2$, the 
            $\Lambda(Q)$-module $Y_{G, N} = \Delta(G)_N$ has projective dimension at most $1$ by \cite[Proposition 5.6.7]{nsw}.
            
            For (i) we consider the short exact sequence $G_{L_\infty, v} \hookrightarrow G_{K_v} \twoheadrightarrow \mathcal{G}_v$
            for a non-archimedean place $v$ of $K$. We apply 
            \cite[Theorems 7.4.1 and 7.1.8]{nsw} to deduce that $G_{K_v}$
            is topologically finitely generated and that
       		$\cd_p G_{K_v} = 2$, respectively. The vanishing of $H_2(G_{L_\infty, v}, \ZZ_p) \simeq H^2(G_{L_\infty, v}, \QQ_p/\ZZ_p)^\vee$ follows from 
            \cite[Propositions 1.5.1 and 7.3.10]{nsw}.
            Now the above observation implies that the local complex
            $\mathcal{L}_v^{\bullet}$ is perfect. The latter is clear
            by construction if $v$ is archimedean. This finishes the proof
            of (i).
            
            Now suppose that the weak Leopoldt conjecture holds and therefore $H_2(H_{\Sigma}, \ZZ_p)$ vanishes. 
            We apply the above argument to the sequence $H_{\Sigma} \hookrightarrow G_{\Sigma} \twoheadrightarrow \mathcal{G}$ in order to deduce the perfectness of the global complex 
            $\mathcal{T}_{\Sigma}^{\bullet} = \mathcal{T}_{S \cup T}^{\bullet}$.
            Note that (as explained in \cite[p.\ 623]{nsw}) \cite[Corollary 10.11.15]{nsw} implies 
            that the Galois group $G_{\Sigma}$ is indeed topologically finitely generated; and \cite[Proposition 10.11.3]{nsw} shows that
            one has $\cd_p G_{\Sigma} \leq 2$, as already mentioned.
            Since $\mathcal{L}_S^{\bullet} = \bigoplus_{v \in S} \Ind_{\mathcal{G}_v}^\mathcal{G} \mathcal{L}_v^{\bullet}$ is a perfect complex of $\Lambda(\mathcal{G})$-modules by (i), 
            so is the main complex $\mathcal{C}_{S, T}^{\bullet}$.
        \end{proof}
        
        In order to study the cohomology of the main complex, we recall the definitions of a few classical $\Lambda(\mathcal{G})$-modules:
        \begin{itemize}
            \item{For each layer $n$ we write $\mathcal{O}_{L_n,S}$ for the
            	ring of $S(L_n)$-integers in $L_n$ and let
            	$\mathcal{O}_{L_n, S, T}^{\times}$ be the subgroup of $\mathcal{O}_{L_n,S}^{\times}$
            	comprising those $S(L_n)$-units which are congruent to $1$ modulo each place in $T(L_n)$. Then
                $E_{S, T} := \varprojlim_n (\ZZ_p \otimes \mathcal{O}_{L_n, S, T}^{\times})$, where the inverse limit is taken with respect to the norm maps, is a finitely generated
                $\Lambda(\mathcal{G})$-module. We denote $E_{S, \emptyset}$ by $E_S$.
            }
            \item{We let
                $\mathcal{Y}_S = \bigoplus_{v \in S} \Ind_{\mathcal{G}_v}^\mathcal{G} \ZZ_p$, where $\ZZ_p$ is endowed with the trivial $\mathcal{G}_v$-action. This module can be identified with the inverse limit $\varprojlim_n \mathcal{Y}_{L_n, S}$ along the $\ZZ_p$-tower $L_\infty/L$ of the $\ZZ_{p}[\mathcal{G}_n]$-modules
                \begin{equation}
                    \mathcal{Y}_{L_n, S} = \bigoplus_{v \in S} \Ind_{(\mathcal{G}_n)_{v(L_n)}}^{\mathcal{G}_n} \ZZ_p \simeq \bigoplus_{w_n \in S(L_n)} \ZZ_p \cdot w_n,
                \end{equation}
	                where the last term is a free $\ZZ_p$-module with $\mathcal{G}_n$-action given by Galois conjugation of places. By construction, $\mathcal{Y}_S$ naturally identifies with $H^1(\mathcal{L}_S^{\bullet})$.
            }
            \item{ We let
                $\mathcal{X}_S$ be the kernel of the augmentation map
                $\mathcal{Y}_S \twoheadrightarrow \ZZ_p$,
                which can be identified with the 
                inverse limit of the 
                $\ZZ_{p}[\mathcal{G}_n]$-modules 
                $\mathcal{X}_{L_n, S} = \ker(\mathcal{Y}_{L_n, S} \twoheadrightarrow \ZZ_p)$.
            }
            \item{
            	We denote the Galois group over $L_\infty$ of its maximal $T$-ramified abelian pro-$p$-extension which is completely split at $S$ by $X_{T, S}^{cs}$.
            }
        \end{itemize}
               
        \begin{prop}
        \label{prop:cohomology_of_complex}
            The complex $\mathcal{C}_{S, T}^{\bullet}$ is acyclic outside degrees 0 and 1 and satisfies the following:
            \begin{enumerate}
                \item{There is a short exact sequence of 
                	$\Lambda(\mathcal{G})$-modules
                \[
                	0 \rightarrow H_2(H_S, \ZZ_p) \rightarrow
                	E_{S,T} \rightarrow H^0(\mathcal{C}_{S, T}^{\bullet})
                	\rightarrow 0.
                \]
                In particular, if the weak Leopoldt conjecture holds for $L_\infty/L$, then
                    ${H^0(\mathcal{C}_{S, T}^{\bullet}) \simeq E_{S, T}}$.
                }
                \item{
                    $H^1(\mathcal{C}_{S, T}^{\bullet})$ fits into a short exact sequence of $\Lambda(\mathcal{G})$-modules
                    \begin{equation}
                    \label{eq:ses_cohomology_one}
                        0 \to X_{T, S}^{cs} \to H^1(\mathcal{C}_{S, T}^{\bullet}) \to \mathcal{X}_S \to 0.
                    \end{equation}
                }
            \end{enumerate}
        \end{prop}
        
        \begin{proof}
            By definition we have a long exact sequence in cohomology
            \[
            	0 \rightarrow H^{0}(\mathcal{C}_{S,T}^{\bullet})
            	\rightarrow H^{0}(\mathcal{L}_S^{\bullet})
            	\rightarrow H^{0}(\mathcal{T}_{S \cup T}^{\bullet}) 
            	\rightarrow
            \]
            \[
            	\rightarrow H^{1}(\mathcal{C}_{S,T}^{\bullet})
            	\rightarrow H^{1}(\mathcal{L}_S^{\bullet})
            	\rightarrow H^{1}(\mathcal{T}_{S \cup T}^{\bullet})
            	\rightarrow H^{2}(\mathcal{C}_{S,T}^{\bullet}) 
            	\rightarrow 0 \dots
            \]
            which immediately implies that $\mathcal{C}_{S,T}^{\bullet}$
            is acyclic outside degrees $0$, $1$ and $2$.
            The map $H^1(\alpha_{S,T}) \colon H^{1}(\mathcal{L}_S^{\bullet})
            \rightarrow H^{1}(\mathcal{T}_{S \cup T}^{\bullet})$ identifies with
            the augmentation map $\mathcal{Y}_S \rightarrow \ZZ_p$,
            which is surjective as $S$ is non-empty. Hence
            $H^{2}(\mathcal{C}_{S,T}^{\bullet})$ vanishes and the kernel
            of $H^1(\alpha_{S,T})$ identifies with $\mathcal{X}_S$.
            Moreover, by local class field theory there is an isomorphism
            \[
            	H^0(\mathcal{L}_S^{\bullet}) \simeq A_S :=
            	\varprojlim_n\bigoplus_{w_n \in S_f(L_n)} A_{w_n}(L_n), \quad
            	A_{w_n}(L_n) := \varprojlim_m (L_n)_{w_n}^{\times} / ((L_n)_{w_n}^{\times})^{p^m},
            \]
            where $S_f \subseteq S$ denotes the set of finite places contained in $S$. 
            Therefore,  the map $H^0(\alpha_{S,T})$ identifies with the natural map
            $A_S \rightarrow X_{S \cup T}$, which has cokernel $X_{T, S}^{cs}$. This finishes the proof of (ii).
			We now claim that we have a commutative diagram of $\Lambda(\mathcal{G})$-modules with exact rows:
            \begin{center}
            	\begin{tikzcd}
            		& & & \bigoplus_{v \in T^p} \Ind_{\mathcal{G}_v}^\mathcal{G} \ZZ_p(1) \arrow[d, hook] & \\
            		 H_2(H_S, \ZZ_p) \arrow[r, hook] \arrow[d, equal] & E_{S,T} \arrow[d, hook] \arrow[r] & A_S \arrow[d, equal] \arrow[r]  & X_{S \cup T} \arrow[d, twoheadrightarrow] \arrow[r, twoheadrightarrow] & X_{T,S}^{cs} \arrow[d, twoheadrightarrow] \\
            		H_2(H_S, \ZZ_p) \arrow[r, hook] & E_{S} \arrow[r] & A_S \arrow[r] & X_{S} \arrow[r, twoheadrightarrow] & X_{\emptyset,S}^{cs} 
            	\end{tikzcd}
            \end{center}
            The bottom sequence is \cite[Theorem 11.3.10(i)]{nsw}, which already
            shows part (i) if $T$ is empty. The set  
            $T^p$ consists of all places $v \in T$ such that 
            $L_w$ contains a primitive $p$-th root of unity for any
            (and hence every) place $w$ of $L$ above $v$. The vertical
            sequence is exact by \cite[Theorem 11.3.5]{nsw}.
            It now suffices to show that an $x \in E_{S}$ has trivial image
            in $X_{S \cup T}$ under the composite map $E_S \rightarrow A_S \rightarrow X_{S \cup T}$
            if and only if $x \in E_{S,T}$; then we are
            done by a diagram chase.
            
            For this consider the left exact sequences
            \[
            	0 \rightarrow \mathcal{O}_{L_n,S,T}^{\times} \rightarrow
            	\mathcal{O}_{L_n,S}^{\times} \rightarrow 
            	\bigoplus_{w_n \in T(L_n)} \kappa(w_n)^{\times},
            \]
            where $\kappa(w_n)$ denotes the residue field of $L_n$
            at $w_n$.
            Taking $p$-completions and the limit along the $\ZZ_p$-tower
            yields an exact sequence
            \[
            	0 \rightarrow E_{S,T} \rightarrow E_S \rightarrow \bigoplus_{v \in T^p} \Ind_{\mathcal{G}_v}^\mathcal{G} \ZZ_p(1),
            \]
            which implies our claim.
        \end{proof}
        
        \begin{remark}
        	If we assume the weak Leopoldt conjecture, then
        	Proposition \ref{prop:cohomology_of_complex} 
        	can also be deduced from Theorem \ref{thm:iso_bks_complex} below. In that particular case, $H_2(H_S, \ZZ_p)$ vanishes and the short exact sequence in point (i) above becomes an isomorphism. 
        \end{remark}

    \section{The complex of Burns, Kurihara and Sano}
    
        In this section we assume that the weak Leopoldt conjecture holds.
        We first recall the complex  employed by Burns, Kurihara and Sano (the ``BKS complex'' in the sequel) in \cite[pp. 1534 and 1539-1540]{bks}. Then we prove that this complex and our complex 
        $\mathcal{C}_{S, T}^{\bullet}$ indeed coincide in the derived category
        of $\Lambda(\mathcal{G})$-modules in all cases where the two settings overlap.
        The differences between the settings are as follows.   
        
           \begin{itemize}
            \item{
                Here $L_\infty/K$ is only assumed to be Galois, whereas it is required to be abelian in \cite{bks}. The definition of the complex therein, however, does not require this hypothesis.
                So we will not impose this condition.
            }
            \item{
                Our condition \textit{$p \neq 2$ if $K$ is not totally imaginary} does not appear in the earlier part of \cite{bks} -- neither in the definition of the complex nor in the formulation of their Main Conjecture. The condition $p \neq 2$ is indeed required in their reformulation of that conjecture \cite[Conjecture 3.14]{bks}.}
        \end{itemize}
        
        We start by briefly recalling the construction of the BKS complex, which is essentially given by an inverse limit of finite-level complexes from \cite{burns_flach}.

        \begin{definition}
        \label{def:complexes_bks}
            \begin{enumerate}
                \item{
                    Let $p$ be a prime number, $F/K$ a finite Galois extension of the number field $K$
                    and $S$ and $T$ two disjoint finite sets of places of $K$ such that $S$ contains $S_\infty \cup S_{\ram}(F/K)$. We set
                    \[
                        \mathcal{B}_{F, S, \emptyset}^{\bullet} = R \Hom_{\ZZ_p}(R \Gamma_c(H_{F, S}, \ZZ_p), \ZZ_p)[-2]
                        \in \mathcal{D}(\ZZ_{p}[\Gal(F/K)]),
                    \]
                    where $H_{F, S}$ denotes the Galois group over $F$ of its maximal $S$-ramified extension and $R \Gamma_c$ denotes the compactly supported cohomology complex (cf. \cite[Equation (3)]{burns_flach}). We then define $\mathcal{B}_{F, S, T}^{\bullet} \in \mathcal{D}(\ZZ_{p}[\Gal(F/K)])$ by the exact triangle
                    \begin{equation}
                    \label{eq:exact_triangle_bks}
                        \mathcal{B}_{F, S, T}^{\bullet} \to \mathcal{B}_{F, S, \emptyset}^{\bullet} \to \bigoplus_{w \in T(F)} \ZZ_p \otimes \kappa(w)^{\times}[0] \to,
                    \end{equation}
                    where the second arrow is as in \cite{bks}.
                }
                \item{
                    For an arbitrary Galois extension $F_\infty/K$, $S \supseteq S_\infty \cup S_{\ram}(F_\infty/K)$ and $T$ disjoint from $S$, we define
                    \[
                        \mathcal{B}_{F_\infty, S, T}^{\bullet} = \varprojlim_F \mathcal{B}_{F, S, T}^{\bullet} \in
                        \mathcal{D}(\Lambda(\Gal(F_{\infty}/K))),
                    \]
                    where $F$ runs over the intermediate finite Galois extensions of $K$, and
                    the limit is taken with respect to the natural transition maps induced by restriction. 
                }
            \end{enumerate}
        \end{definition}

%
    With this notation and that of the previous section, the BKS complex is simply $\mathcal{B}_{L_\infty, S, T}^{\bullet}$, which we abbreviate to $\mathcal{B}_{S, T}^{\bullet}$.
    
	Our next task is to relate these complexes to those of the form considered in the previous section. This is achieved by the following proposition, the idea of whose proof is largely taken from \cite[Theorem 2.4]{nickel_plms}
	(see also \cite[\S 3]{venjakob_on_rw}).
	Since this is crucial for the main result of the article, we repeat
	the argument for convenience of the reader. Indeed we give a slight
	generalisation.

        \begin{prop}
        \label{prop:iso_rgamma_complexes}
            Let $G$ be a profinite group and $N \trianglelefteq_c G$ a closed normal subgroup such that $H_n(N, \ZZ_p) = 0$ for all $n \geq 2$, and set $Q = G/N$. Then there exists an isomorphism of complexes
            \[
                R\Hom(R\Gamma(N, \QQ_p/\ZZ_p), \QQ_p/\ZZ_p) \simeq\ [\stackrel{-1}{Y_{G, N}} \xrightarrow{v} \stackrel{0}{\Lambda(Q)}]
            \]
	            in the derived category $\mathcal{D}(\Lambda(Q))$, where $Y_{G, N} = \Delta(G)_N = \Delta(G)/\Delta(N)\Delta(G)$ and $\upsilon$ is induced by the natural $\ZZ_{p}$-module homomorphism $\Delta(G) \rightarrow \Lambda(Q)$,
	            $g - 1 \mapsto g \pmod N -1$.
        \end{prop}

        \begin{proof}
            Since the left exact functors $(-)^N$ and $\Hom_{\Lambda(N)}(\ZZ_p, -)$ coincide, the complexes $R\Gamma(N, \QQ_p/\ZZ_p)$ and $R\Hom_{\Lambda(N)}(\ZZ_p, \QQ_p/\ZZ_p)$ are naturally isomorphic
            in $\mathcal{D}(\Lambda(Q))$.

            The contravariant functor $\Hom_{\ZZ_p}(-, \QQ_p/\ZZ_p) = (-)^\vee$
            is exact. Noting that $M \ctp_{\Lambda(N)} \ZZ_p \simeq \Hom_{\Lambda(N)}(M, \QQ_p/\ZZ_p)^\vee$ for any compact right $\Lambda(N)$-module $M$ (see the proof of \cite[Corollary 5.2.9]{nsw}), we conclude that
            \begin{equation}
            \label{eq:iso_rg_complex_tensor}
                    R\Hom(R\Gamma(N, \QQ_p/\ZZ_p), \QQ_p/\ZZ_p) \simeq\ \ZZ_p \ctp_{\Lambda(N)}^\LL \ZZ_p
            \end{equation}
            in $\mathcal{D}(\Lambda(Q))$ (where $- \ctp_{\Lambda(N)}^\LL \ZZ_p$ denotes the left derived functor of $- \ctp_{\Lambda(N)} \ZZ_p$) by dualising a projective resolution of $\ZZ_p$ into an injective one of $\QQ_p/\ZZ_p$.

            Now we consider the canonical short exact sequence $\Delta(G) \hookrightarrow \Lambda(G) \twoheadrightarrow \ZZ_p$ of compact $\Lambda(N)$-modules, which induces an exact triangle
            \[
                \Delta(G) \ctp_{\Lambda(N)}^\LL \ZZ_p \to \Lambda(G) \ctp_{\Lambda(N)}^\LL \ZZ_p \to \ZZ_p \ctp_{\Lambda(N)}^\LL \ZZ_p \to
            \]
            in $\mathcal{D}(\Lambda(Q))$ and hence an isomorphism
            \begin{equation}
            \label{eq:zp_derived_cone}
                \ZZ_p \ctp_{\Lambda(N)}^\LL \ZZ_p \simeq\ \Cone(\Delta(G) \ctp_{\Lambda(N)}^\LL \ZZ_p \to \Lambda(G) \ctp_{\Lambda(N)}^\LL \ZZ_p).
            \end{equation}

            The groups $H_i(N, \Lambda(G))$ vanish for all $i \geq 1$, since $\Lambda(G)^\vee \simeq \Map_{cts}(G, \QQ_p/\ZZ_p)$ is $G$-induced and therefore cohomologically trivial. In particular, there are isomorphisms $H_{i + 1}(N, \ZZ_p) \simeq H_i(N, \Delta(G))$ for all $i \geq 1$. This implies $H_i(N, \Delta(G)) = 0$ for all $i \geq 1$ by hypothesis, and hence the complexes $\Delta(G) \ctp_{\Lambda(N)}^\LL \ZZ_p$ and $\Lambda(G) \ctp_{\Lambda(N)}^\LL \ZZ_p$ are both acyclic outside degree 0. But a complex with a single non-trivial cohomology group is isomorphic in the derived category to the complex consisting of that group concentrated in the corresponding degree, i.e.
            \[
                \Delta(G) \ctp_{\Lambda(N)}^\LL \ZZ_p \simeq (\Delta(G) \ctp_{\Lambda(N)} \ZZ_p)[0] = Y_{G, N}[0]
            \]
            and
            \[
                \quad \Lambda(G) \ctp_{\Lambda(N)}^\LL \ZZ_p \simeq (\Lambda(G) \ctp_{\Lambda(N)} \ZZ_p)[0] = \Lambda(Q)[0].
            \]
            This concludes the proof by \eqref{eq:iso_rg_complex_tensor} and \eqref{eq:zp_derived_cone}, since by definition of the cone we have
            \[
                \Cone(Y_{G, N}[0] \to \Lambda(Q)[0]) = [\stackrel{-1}{Y_{G, N}} \xrightarrow{v} \stackrel{0}{\Lambda(Q)}]
            \]
            and the morphisms defining both cones are compatible.
        \end{proof}

        \begin{remark}
        \label{rem:complex_independent_of_g_q}
            It follows that, as a complex of abelian groups, $[Y_{G, N} \to \Lambda(Q)]$ is independent of $G$ and $Q$ up to isomorphism in the derived category -- as long as they fit into a short exact sequence $N \hookrightarrow G \twoheadrightarrow Q$. This should not be too surprising, as by the argument given in \S \ref{subsec:local_and_global_complexes} 
            the cohomology groups of complexes of that form can be shown to be $N^{ab}(p)$, $\ZZ_p$ and $\lbrace 1 \rbrace$.
        \end{remark}

        The above proposition can now be applied to the global and local complexes to obtain isomorphisms
        \begin{equation}
        \label{eq:iso_global_rhom}
            \mathcal{T}_{\Sigma}^{\bullet}\simeq\ R\Hom(R\Gamma(H_{\Sigma}, \QQ_p/\ZZ_p), \QQ_p/\ZZ_p)[-1]
        \end{equation}
        and
        \[
            \mathcal{L}_v^{\bullet} \simeq\ R\Hom(R\Gamma(G_{L_\infty, v}, \QQ_p/\ZZ_p), \QQ_p/\ZZ_p)[-1],
        \]
        where $G_{L_\infty, v}$ is as in \S \ref{subsec:local_and_global_complexes}. Note that the hypotheses
        of the proposition are indeed satisfied. In the global case, this follows from the weak Leopoldt conjecture and $\cd_p G_{\Sigma} \leq 2$ (see the proof of lemma \ref{lem:perfect}); in that of a non-archimedean place $v$, from loc.\ cit.; and in that of an archimedean place $v$, from classical homology of finite groups: if $G_{L_\infty, v}$ is trivial, then its homology in degree $\geq 1$ vanishes for any module; and otherwise, $\lvert G_{L_\infty, v} \rvert = 2$ (so $p$ is odd by assumption) and hence the groups $H_n(G_{L_\infty, v}, \ZZ_p)$ are 2-torsion $\ZZ_p$-modules, i.e. trivial, for all $n \geq 1$.

        In combination with some duality properties, this description of the global and local complexes allows us to prove the main result of this article:

        \begin{theorem}
        \label{thm:iso_bks_complex}
        	Assume the weak Leopoldt conjecture. Then
            there is an isomorphism of complexes $\mathcal{C}_{S, T}^{\bullet} \simeq \mathcal{B}_{S, T}^{\bullet}$ in the derived category $\mathcal{D}(\Lambda(\mathcal{G}))$.
        \end{theorem}

        \begin{proof}
            It is convenient to initially disregard the set $T$ altogether and consider
            \[
                \mathcal{C}_{S, \emptyset}^{\bullet} = \Cone(\mathcal{L}_S^{\bullet} \xrightarrow{\alpha_{S, \emptyset}} \mathcal{T}_S^{\bullet})[-1].
            \]

            Let us consider the local side first. There are isomorphisms
            \begin{align*}
                \mathcal{L}_S^{\bullet}  = \bigoplus_{v \in S} \Ind_{\mathcal{G}_v}^\mathcal{G} \mathcal{L}_v^{\bullet} & \simeq \bigoplus_{v \in S} \Ind_{\mathcal{G}_v}^\mathcal{G} R\Hom(R\Gamma(G_{L_\infty, v}, \QQ_p/\ZZ_p), \QQ_p/\ZZ_p)[-1] \\
                & \simeq \bigoplus_{v \in S} \Ind_{\mathcal{G}_v}^\mathcal{G} R\Hom(\varinjlim_n R\Gamma(G_{L_n, v}, \QQ_p/\ZZ_p), \QQ_p/\ZZ_p)[-1] \\
                & \simeq \varprojlim_n \bigoplus_{v \in S} \Ind_{\mathcal{G}_{n, v}}^{\mathcal{G}_n} R\Hom(R\Gamma(G_{L_n, v}, \QQ_p/\ZZ_p), \QQ_p/\ZZ_p)[-1] \\
                & \simeq \varprojlim_n  R\Hom\big(\bigoplus_{v \in S} \Ind_{\mathcal{G}_{n, v}}^{\mathcal{G}_n} R\Gamma(G_{L_n, v}, \QQ_p/\ZZ_p), \QQ_p/\ZZ_p\big)[-1],
            \end{align*}
            in $\mathcal{D}(\Lambda(\mathcal{G}))$. 
            Here the first isomorphism follows from Proposition \ref{prop:iso_rgamma_complexes}. 
            The second isomorphism is \cite[Proposition 1.5.1]{nsw}. Note that $G_{L_\infty, v} = \bigcap_n G_{L_n, v}$ and this intersection is simply an inverse limit with respect to the canonical embeddings. Even though the cited result only refers to cohomology groups, its proof shows the isomorphism exists on the level of complexes.
            The last two isomorphisms are formal: $R\Hom$ takes colimits (in particular direct limits and sums) in the first component to limits (in particular inverse limits and products) and, in the cases above, $\Ind_{\mathcal{G}_{n, v}}^{\mathcal{G}_n} -$ commutes with $\varprojlim_n$ in the derived category. To see this, note that if $v$ is a non-archimedean place, then $\Ind_{\mathcal{G}_v}^\mathcal{G} -$ applied to a complex amounts to a finite sum of complexes isomorphic to it; and if $v$ is archimedean, then the inverse limit coincides with the complex at the layer $L_0 = L$ (such places do not split in the cyclotomic tower).

            By the same token, the global complex is isomorphic to
            \[
                \mathcal{T}_S^{\bullet} \simeq\ R\Hom(R\Gamma(H_S, \QQ_p/\ZZ_p), \QQ_p/\ZZ_p)[-1] \simeq\ \varprojlim_n R\Hom(R\Gamma(H_{L_n, S}, \QQ_p/\ZZ_p), \QQ_p/\ZZ_p)[-1],
            \]
            where $H_S = \Gal(M_S/L_\infty)$ and $H_{L_n, S} = \Gal(M_S/L_n)$. Recall that $M_S$ was defined as the the maximal $S$-ramified pro-$p$-extension of $L_{\infty}$ -- and hence of $L_n$ for all $n$.

            It follows that $\mathcal{C}_{S, \emptyset}^{\bullet}[2]$ is isomorphic to the cone of a morphism
            \[
                \bigg(\varprojlim_n R\Hom\big(\bigoplus_{v \in S} \Ind_{\mathcal{G}_{n, v}}^{\mathcal{G}_n} R\Gamma(G_{L_n, v}, \QQ_p/\ZZ_p), \QQ_p/\ZZ_p\big)\bigg) \xrightarrow{\alpha} \bigg(\varprojlim_n R\Hom(R\Gamma(H_{L_n, S}, \QQ_p/\ZZ_p), \QQ_p/\ZZ_p)\bigg)
            \]
            and therefore to the inverse limit of the cones of the finite-level morphisms
            \[
                R\Hom\big(\bigoplus_{v \in S} \Ind_{\mathcal{G}_{n, v}}^{\mathcal{G}_n} R\Gamma(G_{L_n, v}, \QQ_p/\ZZ_p), \QQ_p/\ZZ_p\big) \xrightarrow{\alpha_n} R\Hom(R\Gamma(H_{L_n, S}, \QQ_p/\ZZ_p), \QQ_p/\ZZ_p)
            \]
            because cones are defined in terms of direct sums and shifts, both of which commute with inverse limits. The map $\alpha$ is induced by the map $\alpha_{S, T}$ from section \ref{sec:the_main_complex} (with $T = \emptyset$) by definition and hence comes from the inflation and restriction associated to $G_{L_\infty} \twoheadrightarrow H_S$ and $G_{L_\infty, v} \hookrightarrow G_{L_\infty}$ respectively, as these are (after changing base field from $L_\infty$ to $K$) the homomorphisms used in the definition of the local-to-global maps. Note that $R\Hom(-, \QQ_p/\ZZ_p)$ causes a reversal in the direction of the resulting arrows. It follows that $\alpha_n$ is induced from $G_{L_n} \twoheadrightarrow H_{L_n, S}$ and $G_{L_n, v} \hookrightarrow G_{L_n}$.

            The definition of compact-support cohomology as the shifted cone of global-to-local restriction will yield the desired isomorphism by virtue of local and global Tate duality in their derived-categorical form. Namely, \cite[Theorem 4.2.6]{lim} (right column of the diagram therein) shows that the cone of $\alpha_n$ above is isomorphic in $\mathcal{D}(\ZZ_p[\mathcal{G}_n])$ to $R\Hom(R\Gamma_c(H_{L_n, S}, \QQ_p/\ZZ_p), \QQ_p/\ZZ_p)$. We therefore have isomorphisms
            \[
                \mathcal{C}_{S, \emptyset}^{\bullet} \simeq \varprojlim_n R\Hom(R\Gamma_c(H_{L_n, S}, \QQ_p/\ZZ_p), \QQ_p/\ZZ_p)[-2] \simeq \varprojlim_n R\Gamma(H_{L_n, S}, \ZZ_p(1))[1],
            \]
            in $\mathcal{D}(\Lambda(\mathcal{G}))$,
            the last one being global Tate duality (bottom row of the diagram in the cited theorem). Lastly, we apply Nekovář duality. Although introduced in \cite{nekovar}, a closer formulation to our setting is given by Lim and Sharifi in \cite[Theorem, p.\ 623]{lim_sharifi}. 
            We choose $R = \ZZ_p$, dualising complex $\omega_R = \ZZ_p[0]$ and complex $T = \ZZ_p(1)[0]$. 
            It is then easy to see that $R\Hom_{\ZZ_p}(T, \omega_R)$ is represented by the Kummer dual $T^{\ast} = \ZZ_p(-1)[0]$ of $T$, 
            and thus the aforementioned theorem implies that
            \begin{equation}
            \label{eq:rgamma_finite_level}
                R\Gamma(H_{L_n, S}, \ZZ_p(1)) \simeq\ R\Hom_{\ZZ_p}(R\Gamma_c(H_{L_n, S}, \ZZ_p), \ZZ_p)[-3]
            \end{equation}
            in $\mathcal{D}(\ZZ_p[\mathcal{G}_n])$ for all $n \geq 0$ and hence
            \begin{equation}
            \label{eq:iso_inverse_zp(1)}
                \mathcal{C}_{S, \emptyset}^{\bullet} \simeq \varprojlim_n R\Gamma(H_{L_n, S}, \ZZ_p(1))[1] \simeq \varprojlim_n R\Hom_{\ZZ_p}(R\Gamma_c(H_{L_n, S}, \ZZ_p), \ZZ_p)[-2] = \mathcal{B}_{S, \emptyset}^{\bullet}
            \end{equation}
			in $\mathcal{D}(\Lambda(\mathcal{G}))$.
            It remains to bring $T$ into the discussion. The difference between $\mathcal{C}_{S, T}^{\bullet}$ and $\mathcal{C}_{S, \emptyset}^{\bullet}$ is given by the following exact triangle in $\mathcal{D}(\Lambda(\mathcal{G}))$:
            \[
                \mathcal{C}_{S, T}^{\bullet} \to \mathcal{C}_{S, \emptyset}^{\bullet} \to \bigoplus_{v \in T^p} \Ind_{\mathcal{G}_v}^\mathcal{G} \ZZ_p(1)[0] \to,
            \]
            where as before $T^p$ denotes set of all places $v \in T$ such that 
            $L_w$ contains a primitive $p$-th root of unity for any
            (and hence every) place $w$ of $L$ above $v$.
            This follows from the definition of the complex and
            \cite[Theorem 11.3.5]{nsw}.

            As to the complex of Burns, Kurihara and Sano, the limit of \eqref{eq:exact_triangle_bks} along the cyclotomic tower yields an exact triangle
            \[
                \mathcal{B}_{S, T}^{\bullet} \to \mathcal{B}_{S, \emptyset}^{\bullet} \to \varprojlim_n \bigoplus_{w_n \in T(L_n)} \ZZ_p \otimes \kappa(w_n)^{\times}[0] \to.
            \]
            In the last term, the limit is taken with respect to the norm maps (cf. \cite[p.\ 1540]{bks}) and  is precisely $\bigoplus_{v \in T^p} \Ind_{\mathcal{G}_v}^\mathcal{G} \ZZ_p(1)[0]$. The result now follows by comparing the last two exact triangles and noting that the isomorphisms between their second and third terms are compatible because they come from the natural dualities and transition maps.
        \end{proof}
 
    \nocite*
    \bibliography{references}{}
    \bibliographystyle{amsalpha}

\end{document}